\documentclass[12pt]{article}

\usepackage{cmap}  % should be before 'fontenc'
\usepackage[T2A]{fontenc}
\usepackage[utf8]{inputenc}
\usepackage[english]{babel}
\usepackage[usenames]{color}
\usepackage{amsmath,amssymb,amsthm}
\usepackage[pdftex,colorlinks=true,linkcolor=blue,urlcolor=red,unicode=true,hyperfootnotes=false,bookmarksnumbered]{hyperref}
\usepackage{mathtools}
\usepackage{xcolor}

\newcommand{\eps}{\varepsilon}

\renewcommand{\le}{\leqslant}
\renewcommand{\ge}{\geqslant}

\newcommand{\aaa}{\mathcal{A}}
\newcommand{\bb}{\mathcal{B}}
\newcommand{\ff}{\mathcal{F}}
\newcommand{\fg}{\mathcal{G}}
\newcommand{\fh}{\mathcal{H}}
\newcommand{\hh}{\mathcal{H}}
\newcommand{\I}{\mathcal{I}}

\newcommand{\E}{\mathsf{E}}

\newtheorem{thm}{Theorem}
\newtheorem{prob}[thm]{Problem}
\newtheorem{cla}[thm]{Claim}
\newtheorem{lemma}[thm]{Lemma}
\newtheorem{cor}[thm]{Corollary}
\newtheorem{prop}[thm]{Proposition}

\title{Trivial colors in colorings of Kneser graphs}
\author{Sergei Kiselev\footnote{Moscow Institute of Physics and Technology, Email: {\tt kiselev.sg@gmail.com}}, Andrey Kupavskii\footnote{G-SCOP, CNRS, University Grenoble-Alpes, France and Moscow Institute of Physics and Technology, Russia; Email: {\tt kupavskii@yandex.ru}.  
}
}
\begin{document}

\maketitle
\begin{abstract}
    We show that any proper coloring of a Kneser graph $KG_{n,k}$ with $n-2k+2$ colors contains a trivial color (i.e., a color consisting of sets that all contain a fixed element), provided $n>(2+\varepsilon)k^2$, where $\varepsilon\to 0$ as $k\to \infty$. This bound is essentially tight. This is a consequence of a more general result on the minimum number of non-trivial colors needed to properly color $KG_{n,k}$.
\end{abstract}

\section{Introduction}
Throughout the paper, we use standard notations $[n]:=\{1,\ldots,n\},$ $[a,b]=\{a,a+1,\ldots,b\}$, $2^{X}$ for the power set of $X$, and ${X\choose k}$ for the collection of all $k$-element subsets of $X.$ Any $\ff\subset 2^X$ we call a {\it family.}

Given positive integers $n\ge 2k$, a {\it Kneser graph} $KG_{n,k}$ is a graph whose vertex set is the collection of all $k$-element subsets of the set $\{1,\ldots, n\}$, with edges connecting pairs of disjoint sets. One of the classical results in combinatorics, conjectured by Kneser \cite{Knes} and proved by  Lov\'asz \cite{Lova}, states that the chromatic number of $KG_{n,k}$ is equal to $n-2k+2$. The proof of Lov\'asz, as well as subsequent proofs given by B\'ar\'any and Green, rely on the Borsuk--Ulam theorem and thus on combinatorial topology. There was a ``combinatorial'' proof given by Matousek and Ziegler, which used Tucker's lemma instead of the Borsuk-Ulam theorem, but it essentially uses the same machinery as the previous proofs. One of the drawbacks of only having a combinatorial topology proof is that the approach is very sensitive to the setting, and many related questions seem to be out of reach of that method. In particular, we do not know, how big is the largest vertex subset of $KG_{n,k}$ that we can properly cover in $n-2k+1$ colors. This motivates the quest of searching for a more extremal-combinatorial approach to this question. One step in this direction is to better understand the structure of proper colorings of $KG_{n,k}$ that use few (minimum possible number of) colors.    

Note that each color in the coloring of $KG_{n,k}$ forms an independent set, which, in turn, is an {\it intersecting family:} i.e., a family of sets in which any two intersect. In this terminology, a proper coloring of $KG_{n,k}$ into $t$ colors is  the same as a partition  of ${[n]\choose k}$ into $t$ intersecting families. We say that an intersecting family is {\it trivial}, or a {\it star}, if all sets in the family contain a fixed element $i$. If this is the case, then we say that $i$ is a {\it center} of $\ff$. More generally, we call two families $\aaa,\bb\subset 2^X$ {\it cross-intersecting,} if $A\cap B\ne \emptyset$ for any $A\in \aaa, B\in \bb$. 

The standard example of a proper coloring of $KG_{n,k}$ consists of $n-2k+1$ star  $\mathcal C_1,\ldots, \mathcal C_{n-2k+1}$, where $\mathcal C_i:=\{A\in {[n]\choose k}: A\subset [i,n], i\in A\}$, and the set ${[n-2k+2,n]\choose k}$.  That is, all but one colors are trivial.  The advantage of always having a trivial color in a coloring of $KG_{n,k}$ is that one can remove the color, remove the corresponding center from the ground set, and thus reduce the study of colorings of $KG_{n,k}$ to that of $KG_{n-1,k}.$   This motivates the following question, asked by the authors in \cite{KK}:

\begin{prob}\label{prob1}
Given $k$, what is the largest number $n = n(k)$, such that there exists a proper coloring of $KG_{n,k}$ into $n-2k+2$ colors without any trivial colors?
\end{prob}

Let us recall two important results from extremal set theory. First, the Erd\H os--Ko--Rado theorem~\cite{EKR}  states that $\alpha(KG_{n,k}) = {n-1\choose k-1},$ provided $n\ge 2k$. The Hilton--Milner theorem~\cite{HM} states that any non-trivial intersecting family of $k$-subsets of $[n]$ has size at most ${n-1\choose k-1}-{n-k-1\choose k-1}+1$, provided $n > 2k$. Let us use the latter to establish a simple bound on $n(k).$

\begin{prop}\label{propintro}
We have $n(k) < k^3$ for $k \ge 3$.
\end{prop}

\begin{proof}
The Hilton-Milner theorem states that  any non-trivial intersecting family in ${[n]\choose k}$ has size at most $\binom{n-1}{k-1} - \binom{n-k-1}{k-1} + 1\le k\binom{n-2}{k-2}$. If we have a coloring of $KG_{n,k}$ with non-trivial colors, then we need at least $\binom{n}{k} / \left(k\binom{n-2}{k-2}\right) = \frac{n(n-1)}{k^2(k-1)}$ colors cover all vertices. This is larger than $n-2k+2$ for $n=k^3$ and $k\ge 3$.
\end{proof}
Doing  a bit more careful calculations, one can show that $n(k) \le k^3 - k^2 - 2k + 2$ for $k\ge 3$ using the argument above. However, the $k^3$-barrier for $n$ is not easy to improve.

On the other hand, in \cite{KK} we provided a construction of a coloring with non-trivial colors, which gives $n(k) \ge 2(k-1)^2$ for $k \ge 3$. For completeness, we will reproduce the construction in the next section.

After the first version of the paper appeared, it was pointed out to us that a similar question, but for a very different parameter range, was asked by Katona and partially answered by Sanders~\cite{Sanders}:

\begin{prob}\label{prob2}
Given $n$ and $k$, what is the smallest number $m = m(n, k)$, such that there exists a proper coloring of $KG_{n,k}$ into $m$ colors without any trivial colors?
\end{prob}

Sanders proved the following result. 
\begin{thm}[\cite{Sanders}] \label{thm:Sanders}
There is an absolute constant $C>0$ such that for any fixed $k \ge 3$, 
$$
m(n, k) \ge
\frac{n^2}{2k(k-1)} \left(1 - \frac{Ck^{3.5} e^k}{n}\right),
$$
$$
m(n, k) \le
\frac{n^2}{2k(k-1)} + O(n)
\quad\mbox{when}\quad n > k^2.
$$
\end{thm}
{\bf Remark. } In \cite{Sanders}, the exact form of the lower bound is stated in the proof of Proposition~2, and the subtracted term has the form  $\frac{k^{k+1}(k-2)}{2n(k-1)(k-3)!}$.  %with $O(n)$-term instead of $2n.$ The upper bound wasn't stated in this form explicitly in \cite{Sanders}. For the lower bound in this form see the proof of Proposition 2 in \cite{Sanders}.)
\vskip+0.1cm
The lower bound in this theorem is meaningful only for very large $n$, i.e., for $n = \Omega(k^{3.5} e^k)$, and  this theorem does not give an answer to  Problem~\ref{prob1}. 

The main goal of this paper was to give a satisfactory answer to the question asked in Problem~\ref{prob1}. For the sake of comparison, we also extended our result so that it has the same form as the result of Sanders.   %but after comments from an anonymous referee we noted that our methods allow to improve bounds on Problem~\ref{prob2} as well. 

\begin{thm}\label{thmmain3}
Let $n$, $k$ be sufficiently large integers such that $n < e^{k^{0.1} / 10}$. Put $\eps = \eps(n, k) := \frac{5}{\ln k}$ and $m := \frac{n^2}{2k(k-1)}\big(1-\frac{\eps k^2}n\big)$. Then $\binom{[n]}{k}$ cannot be covered by $m$ intersecting non-trivial families. In particular, $m(n, k) > \frac{n^2}{2k(k-1)}\big(1-\frac{\eps k^2}n\big)$ for all sufficiently large $n$ and $k$.
\end{thm}
The restriction $n < e^{k^{0.1} / 10}$ seems to be a technical artifact of the proof. It can be easily improved to $n<e^{k^{0.49}}$, and it is likely there is a way to get rid of it completely. We decided to avoid cramming the text with additional details needed to get rid of this condition, since we think that the case of $n$ that is comparable to $k$ is the most interesting one.

\begin{cor}
Let $n, k$ be sufficiently large integers, and $\eps = \eps(k) := \frac{5}{\ln k}$ be such that $n \ge (2 + \eps)k^2$. Then $\binom{[n]}{k}$ cannot be covered by $n$ intersecting non-trivial families. In particular, $n(k)<(2 + \eps)k^2 $ for all sufficiently large $k$.
\end{cor}

Note that the theorem is stated in terms of covering ${[n]\choose k}$ by intersecting families. This is slightly stronger than the same statement for partitions of ${[n]\choose k}$  (i.e., colorings of $KG_{n,k}$) since any partition is a covering as well.

The methods that we use bear some superficial resemblance to those of Sanders. Similarly to \cite{Sanders}, we construct some graph from a given coloring and work with independent $k$-sets of that graph. However, our way to construct the graph and work with its independent $k$-sets is different and more subtle. In the core of the proof, we combine a graph-theoretic result of Khad\v ziivanov and Nikiforov \cite{KhN} with a certain subtle decomposition of intersecting families. This approach may be useful for other problems related to intersecting families.

We note that our results are in line with the results that deal with the following problem of Erd\H os: for given $n,k,t$, what is the largest size of a family $\ff\subset {[n]\choose k}$ such that $\ff$ is a union of at most $t$ intersecting families? Note that the result that $\chi(KG_{n,k})\ge n-2k+2$ in this language states that $|\ff|<{n\choose k}$ if $t<n-2k+2$. The natural conjecture here is that, for most triples of parameters, the extremal example is a union of $t$ stars. This was shown for $t=2$ and $n>ck$ with $c\approx 2.62$ by Frankl and F\"uredi \cite{FF} and for any constant $t$ and $n>2k +C_t k^{2/3}$ by Ellis and Lifshitz \cite{EL}. Frankl and F\"uredi also provided examples when this is not true. Our result does not imply anything for this problem directly, but the proof technique allows for obtaining results in a very different regime when $n$ is quadratic in $k$, but $t$ is allowed to be large (even close to $\chi(KG_{n,k})$. This will be investigated in a subsequent paper. We also note that these questions are related to the Erd\H os Matching Conjecture (cf. \cite{FK} for an up-to-date account of the problem and related questions).

The rest of the paper is structured as follows. In the next section, we shall present the lower bound construction from \cite{KK}. In Section~\ref{sec3}, we will prove a slightly weaker upper bound on $n(k)$. In Section~\ref{sec4}, we will prove Theorem~\ref{thmmain3}. 

\section{Lower bound construction}\label{sec2}
\begin{prop}[\cite{KK}]
  We have $n(k)\ge 2(k-1)^2$ for $k\ge 3$.
\end{prop}
 \begin{proof}Put $n:=2(k-1)^2$ and split the ground set $[n]$ into $k-1$ blocks $A_i$ of size $2k-2$. For each block, say, $A_1:=[2k-2]$, consider the following covering by intersecting families: for $i=1,\ldots, 2k-5$, define
$$F_i:=\{2k-4,2k-3,2k-2\}\cup \{i+1,\ldots, i+k-3\},$$
where the addition and subtraction in the second part of the set is modulo $2k-5$ (and thus the elements belong to $[2k-5]$). Consider the intersecting Hilton-Milner-type families of the form
$$\mathcal H_i:=\Big\{F\in {[n]\choose k}: i\in F, F\cap F_{i}\ne\emptyset\Big\}\cup \{F_{i}\}.$$
Complement it with the intersecting family
$$\mathcal G:=\Big\{F\in {[n]\choose k}: |F\cap \{2k-4,2k-3,2k-2\}|\ge 2\Big\}.$$
If a set $G\cap [2k-2]\supset \{i,j\}$ for $i<j$, then $G$ is contained in one of the families $\mathcal G$, $\mathcal H_l,$ $l\in [2k-5]$. Indeed, \begin{itemize}
\item if $\{i,j\}\subset \{2k-4,2k-3,2k-2\}$, then $G\subset \mathcal G$;
\item if $i<2k-4\le j$, then $G\subset \mathcal H_i$;
\item if $j<2k-4$ and $j\le i+k-3$, then $G\subset \mathcal H_i$; 
\item if $i+k-2\le j<2k-4$, then $j+k-3$ mod $2k-5$ is at least $i$, and $G$ is contained in $\mathcal H_j$.
\end{itemize}
Therefore, any set intersecting $A_1$ in at least $2$ elements is contained in one of the intersecting families given above. On the other hand, any $k$-set must intersect one of the $k-1$ blocks in at least $2$ elements. Thus, considering similar collections of intersecting families in the other blocks, we get that the whole of ${[n]\choose k}$ is covered.

We have $2k-4$ intersecting families on each block, which gives $(2k-4)(k-1)$ families in total. On the other hand, $\chi(KG_{n,k}) = 2(k-1)^2-2k+2 = 2(k-2)(k-1)$, that is, the number of intersecting families we used equals  the chromatic number of the graph. It is also clear that none of the families is a star, and we can easily preserve this property when making a coloring (rather than a covering).
\end{proof}

We note that most of the families in the coloring presented above are Hilton--Milner type families.

\section{A weaker upper bound}\label{sec3}
We say that $C\subset X$ is a {\it cover} of a family $\ff\subset 2^X$ if $C\cap A\ne \emptyset$ for any $A\in \ff$. Let $\tau(\ff)$ stand for the size of the smallest cover of $\ff.$ Note that saying that $\tau(\ff)\ge 2$ is the same as saying that $\ff$ is non-trivial. 
In this section we are going to prove the following theorem.
\begin{thm}\label{thmmain1}
Consider $m$ intersecting families $\ff_1, \ldots, \ff_m \subset \binom{[n]}{k}$ with $\tau(\ff_i) \ge 2$ and such that $\ff_1\cup\ldots\cup\ff_m = \binom{[n]}{k}$. Then $m > n^2 / (8k^2)$.
\end{thm}
A simple numerical corollary is as follows.

\begin{cor}
If $n\ge 8k^2$, $k\ge 2$ and $KG_{n,k}$ is covered by $n-2k+2$ intersecting families, then one of these families is trivial. In short, $n(k)<8k^2.$
\end{cor}

In the proof of Theorem~\ref{thmmain1}, we are going to use the following elegant result due to Spencer. We reproduce its proof for completeness. Note that $\E$ stand for the expectation, and for a family $\hh$ and a set $A$, we denote by $\hh[A]$ the restriction of $\hh$ on $A$, that is, $\hh[A] = \{X\in \hh: X\subset A\}$. Given a family $\ff\subset 2^X$, a subset $I\subset X$ is an {\it independent set} in $\ff$ if no set from $\ff$ is entirely contained in $I$. 
\begin{thm}[\cite{S}]
Consider a family $\fh\subset 2^{[n]}$ containing no independent set of size $b$ and let~$\hh^{(k)}$, $\hh^{(k)}\subset \hh,$ be the family of all $k$-sets of $\hh$, $1\le k \le n$. Then for any $0 < p < 1$ the following holds:
$$
\sum_{i=2}^n |\fh^{(i)}|p^i > np - b.
$$
\label{spencer}
\end{thm}

\begin{proof}
Take a random subset $A$ of $[n]$, including each element independently with probability $p$. Then, for each set $X \in \hh[A],$ remove an arbitrary element $v\in X$ from $A$. The resulting set $A'$ is clearly independent, therefore $|A| - |\fh[A]| \le |A'| < b$. Since this holds for any $A$, we get that the same holds on average: $$\E |A| - \E |\fh[A]| < b.$$
The statement of the theorem follows from the inequality above by substituting the values of the expectations: $\E |A| = np$ and $\E |\fh[A]| = \sum_{i=2}^n |\fh^{(i)}|p^i$.
\end{proof}

Note that the bound in Theorem~\ref{spencer} is not sharp, e.g. for a 2-graph it states that the number of edges is at least $n^2/(4b)$, while Tur\'an's theorem gives approximately $n^2/(2b)$ edges. This is a potential direction for improvement, which, unfortunately, does not give results as sharp as Theorem~\ref{thmmain3}.

We say that a family $\fh$ \emph{set-covers} a family $\ff$ if for each set $F\in\ff$ there is a set $H\in\fh$ such that $H \subset F$. Note that any independent set in $\fh$ is also independent in $\ff$.

\begin{lemma}
Let $\ff \subset \binom{[n]}{k}$ be an intersecting family with $\tau(\ff) = \tau$. Then it can be set-covered by a family of $\tau$-sets of size $\tau k^{\tau-1}$.
\label{covers_lemma}
\end{lemma}
\begin{proof}
For any $Y\subset [n]$, we use a standard notation $\ff(Y):=\{F\in \ff: Y\subset F\}$. Consider a cover $X$ of $\ff$ of size $\tau$. Define $\fh_1\subset {[n]\choose 1}$ as follows: $\fh_1:=\{\{i\}: i\in X\}.$ Then $\ff\subset \cup_{G\in \fh(1)}\ff(G)$ by the definition of a cover.

For each  $1\le \ell<\tau$, let us show how to construct $\fh_{\ell+1}$ from $\fh_\ell$. More precisely, assume that we have a family $\fh_\ell\subset {[n]\choose \ell}$ of at most $\tau k^{\ell-1}$ sets such that $\ff\subset \cup_{G\in \fh_{\ell}}\ff(G)$. For each set $G\in \fh_\ell$, consider a set $F_G\in \ff$ that is disjoint with $G$. Such set must exist since $|G|<\tau$. Put $\fh_{\ell+1}:=\{G\cup \{i\}: G\in \fh_{\ell}, i\in F_G\}$. It should be clear that $|\fh_{\ell+1}|\le \tau k^{\ell}$ and that $$\ff\subset \bigcup_{G'\in \fh_{\ell+1}}\ff(G').$$ 
Finally, we put $\fh:=\fh_\tau$.
\end{proof}

\begin{proof}[Proof of Theorem~\ref{thmmain1}]
Consider a collection $\ff_1,\ldots,\ff_m$ of intersecting families that cover ${[n]\choose k}$. Put $\tau_i := \tau(\ff_i)$ and let $\fh_i$ be a set-covering of $\ff_i$ from Lemma~\ref{covers_lemma}. Then the union $\fh_1\cup\ldots\cup\fh_m$ has no independent set of size $k$ and from Theorem~\ref{spencer} we have
$$
\sum_{i=1}^m \tau_i k^{\tau_i-1} p^{\tau_i} > np - k
$$
which is equivalent to 
$$
\sum_{i=1}^m \tau_i (kp)^{\tau_i-1} > n - k/p.
$$

Note that $x a^{x-1} < 2a$ for $x > 2$ and $a < 1/\sqrt{e}$, therefore the inequality above is implied by the following inequality:
$$
2m kp > n - k/p,
$$
provided that $kp<1/\sqrt e$. This is true for our choice of $n$ if we take $p = 2k/n$. Substituting this value of $p$ in the last displayed inequality, we get $4mk^2/n>n/2,$ which is equivalent to the statement of the theorem.
\end{proof}

\section{Proof of Theorem~\ref{thmmain3}}\label{sec4}

Put $m = \frac{n^2 - \eps nk^2}{2k(k-1)}$ and
assume that there are intersecting families $\ff_1, \ldots, \ff_m$ with $\tau(\ff_i) \ge 2$ such that $\ff_1\cup \ldots\cup \ff_m = \binom{[n]}{k}$.

The first step of the proof is to split each color $\ff_i$ into two parts such that each is easier to deal with. 
This is done using the following simple lemma.

\begin{lemma}\label{lempart}
Let $\fg$ be an intersecting family of $k$-sets with $\tau(\fg) \ge 2$. Then we can split it into $\fg'\sqcup\fg''$, where $\fg'$ can be set-covered by at most $k$ 2-edges and $\fg''$ cross-intersects some family of $t$-sets $\fg^\times$ with $t\in\{k-1,k\}$ and $\tau(\fg^\times) \ge \sqrt{k}$.
\end{lemma}

\begin{proof}

If $\tau(\fg) \ge \sqrt{k}$, we can put $\fg'' := \fg$, since $\fg$ cross-intersects itself. So we will assume that $2 \le \tau(\fg) < \sqrt{k}$.

Put $\tau := \tau(\fg)$ and let $a_1, \ldots, a_\tau$ be a piercing set of $\fg$. Put $\fg_i = \{G\setminus\{a_i\} \colon G\in\fg, a_i\in G\}$.
Note that we can set-cover $\fg$ by $S := \sum_{i=1}^\tau \tau(\fg_i)$ edges: for each $i$ we draw edges from $a_i$ to vertices of a piercing set of $\fg_i$.

If $S \le k$ we can put $\fg' = \fg$ and $\fg'' = \varnothing$. Otherwise, for some $i$ we have, $\tau(\fg_i) > k / \tau > \sqrt{k}$.
Then we put $\fg'$ to be the family of all sets containing $a_i$, $\fg'' := \fg\setminus\fg'$ and $\fg^\times := \fg_i$.
\end{proof}

Using this lemma, we split $\ff_i = \ff'_i\sqcup\ff''_i$ for each $i\in[m]$.
Next, consider a ($2$-)graph $H$ on $[n]$ formed by at most $km$ edges that altogether set-cover $\ff'_1,\ldots, \ff'_m$ (the existence of such $H$ is guaranteed by the lemma). 
Since $mk < \frac{n(n-k+1)}{2(k-1)}$, by Tur\'an's theorem $H$ has at least one independent $k$-set.
Each independent $k$-set of $H$ should belong to some $\ff_i$, therefore it must be contained in the corresponding family $\ff_i''$.

%\textbf{Remark.} Tur\'an's theorem here seems to be the main reason, why we can't replace $2k(k-1)$ in denominator of ``$m = \frac{n^2 - \eps n k^2}{2k(k-1)}$'' with something smaller, e.g. $2k(k-2)$. (Это довольно дурацкое замечание, но знаменатель становится важен при больших $n$ и не сразу очевидно, почему его нельзя сделать меньше)

The following theorem is the crux of the proof.
\begin{thm}\label{thmkey}
In the assumptions of Theorem~\ref{thmmain3}, let $G$ be a graph on $[n]$ with at most $mk$ edges and
let $\mathcal I(G)$ be the family of all $k$-sets in $[n]$ that are independent in $G$. Let $\ff\subset \mathcal I(G)$ be a family of independent $k$-sets in G, which cross-intersects some family $\fg$ of $t$-sets with $\tau(\fg) \ge \sqrt{k}$ and $t \le k$. Then $|\mathcal I(G)| > m\cdot |\ff|$.
\end{thm}
We prove this theorem in a separate subsection. Using this result, it is straightforward to finish the proof of Theorem~\ref{thmmain3}. Indeed, the $k$-sets that are not set-covered by $H$ ($\mathcal I(G)$ in the notation of Theorem~\ref{thmkey}) must be contained in the union of $\ff''_i,$ $i\in [m].$ However by Lemma~\ref{lempart}, each $\ff''_i$ cross-intersects a family of $t$-sets $\fg_i$, where $t\le k$, with $\tau(\fg_i)\ge \sqrt k$. Theorem~\ref{thmkey} then guarantees that $\mathcal I(H)$ satisfies $|\mathcal I(H)|>m|\ff''_i|,$ and thus $\mathcal I(H)$ cannot be covered by the union of $\ff''_i$, a contradiction.

\subsection{Proof of Theorem~\ref{thmkey}}

We are going to use the following result by Khad\v ziivanov and Nikiforov \cite{KhN} (see \cite{R} for a reformulation and a proof in English):
\begin{thm}
For a given graph $G$ let $\gamma$ be the density $\frac{|E(G)|}{|V(G)|^2}$ and $N_r(G)$ be the number of cliques on $r$ vertices, $r \le |V(G)|$. Then, if $\gamma \ge \frac{r-2}{2(r-1)}$, we have
\begin{equation}
N_r(G) \ge \frac{2(r-1)\gamma - (r-2)}{r} \cdot |V(G)| \cdot N_{r-1}(G)
\mbox{~~~and~~~}
N_{r-1}(G) > 0.
\label{reiher_eq}
\end{equation}
\label{reiher}
\end{thm}

Note that if the inequality on $\gamma$ in  Theorem~\ref{reiher}  holds for some $r$, then it holds for smaller $r \ge 2$ as well. That is, we can apply $\eqref{reiher_eq}$ several times to compare $N_r(G)$ with $N_{r'}(G)$ for $r > r' \ge 2$.

The idea behind the proof is as follows. Similarly to the proof of Lemma~\ref{covers_lemma} we inductively construct a family that set-covers $\ff$ of larger and larger uniformity using sets from $\fg$. (These set-covers are encoded in the families $\mathcal H_\ell$ from the proof. We have to be more careful with these set-covers than in Lemma~\ref{covers_lemma}, in order to ensure disjointness of certain families.) At each step of the procedure, indexed by the uniformity of $\mathcal H_\ell$, we consider the collection of independent sets that is set-covered by $\mathcal H_\ell$, and bound from below the proportion of these sets that would be ``missed out'' by the set-cover $\mathcal H_{\ell+1}$ of size $1$ larger. We use the result of Khad\v ziivanov and Nikiforov for this bound.

\begin{proof}[Proof of Theorem~\ref{thmkey}]
Let $\I$ be the family of all independent $k$-sets in $G$. Put $\tau := \tau(\fg)$. For $A\subset [n]$, $|A| < \tau$, fix a set $F_A\in\fg$ which is disjoint with $A$. For $A \subset B \subset [n]$ and a family $\mathcal W\subset 2^{[n]}$, denote $$\mathcal W(A, B) := \{F \setminus A\colon F\in\mathcal W, F\cap B = A\}.$$

Put $N_\ell = t^\ell$. First, we are going to construct a sequence of families $\fh_\ell$ of pairs $(A_i, B_i)$, $i=1,\ldots, N_\ell$, such that $\ff = \bigsqcup_{i=1}^{N_\ell} \ff(A_i, B_i)$, $\I \supset \bigsqcup_{i=1}^{N_\ell} \I(A_i, B_i)$, $|A_i| = \ell$, $|B_i| \le \ell k$. (Disjointness is crucial here.)
\vskip+0.2cm

{\bf Construction of $\mathcal H_\ell$.}
First, we put $\fh_0 := \{(\varnothing, \varnothing)\}$.

For each  $0\le \ell < \sqrt{k} \le\tau$, let us show how to construct $\fh_{\ell+1}$ from $\fh_\ell$. We have $\fh_\ell := \{(A_1, B_1), \ldots, (A_{N_\ell}, B_{N_\ell})\}$ such that $|A_i| = \ell$,
$$
\ff = \bigsqcup_{i=1}^{N_\ell} \ff(A_i, B_i)
\mbox{~~~and~~~}
\I \supset \bigsqcup_{i=1}^{N_\ell} \I(A_i, B_i).
$$

For each pair $(A, B) \in \fh_\ell$ put $\{f_1, \ldots, f_t\} := F_A$. Since each set in $\ff$ intersects $F_A$, we can decompose
$$
\ff(A, B) = \bigsqcup_{j=1}^t \ff\big(A\cup\{f_j\}, B \cup \{f_1, \ldots, f_j\}\big)
$$
and 
\begin{equation}
\I(A, B) = \I(A, B \cup F_A) \sqcup \bigsqcup_{j=1}^t \I(A\cup\{f_j\}, B \cup \{f_1, \ldots, f_j\})
\label{indep_decomp}
\end{equation}

Then we put
$$\fh_{\ell+1} := \bigcup_{(A,B)\in \fh_\ell}\Big\{ \big(A\cup \{f_j\},B\cup \{f_1,\ldots, f_j\}\big)\colon j=1,\ldots, t, \  \text{where}\ \{f_1,\ldots, f_t\}=F_A \Big\}.$$
This completes the construction of $\mathcal H_\ell.$\\

Put $\I_\ell := \bigsqcup_{(A, B)\in\fh_\ell} \I(A, B)$. Note that, since $\ff(A, B) \subset \I(A, B)$, we have $\ff \subset \I_\ell$ for any $\ell\le \tau$.

Let $c_\ell \ge 0$ be such that for any $(A, B)\in\fh_\ell$,  $\{f_1,\ldots,f_t\} = F_A$, we have
\begin{equation}\label{eqcl}
\I(A, B\cup F_A) \ge
c_\ell \cdot \sum_{j=1}^t \big|\I(A\cup\{f_j\}, B\cup\{f_1,\ldots,f_j\})\big|.
\end{equation}
Note that by the definition of $c_\ell$ and \eqref{indep_decomp} we have
$$
|\I_\ell| \ge (1 + c_\ell)\cdot|\I_{\ell + 1}|
$$
for any $0\le \ell < \sqrt{k} \le \tau$ and, therefore,
\begin{equation}
|\I| = |\I_0| \ge |\I_\tau|\cdot \prod_{\ell=0}^{\sqrt{k}-1} (1 + c_\ell) \ge
|\ff|\cdot \prod_{\ell=0}^{\sqrt{k}-1} (1 + c_\ell).
\label{c_product}
\end{equation}

\medskip

Now we are going to show that we can put $c_\ell := e^{-1/\eps}$ in \eqref{eqcl}.

Consider a pair $(A, B)$ from $\fh_\ell$ and put $\{f_1, \ldots, f_t\} = F := F_A$. For convenience, put $\I' := \I(A, B)$. We can decompose
$$
\bigsqcup_{j=1}^t \I'(\{f_j\}, \{f_1, \ldots, f_j\}) = \bigsqcup_{\varnothing\ne A'\subset F} \I'(A', F).
$$

Let us compare |$\I'(\varnothing, F)$| and $|\I'(A', F)|$, for $A' \subset F$, $A' \ne \varnothing$ using Theorem~\ref{reiher}.

Put $X:=[n]\setminus (B\cup F)$. We are going to apply Theorem~\ref{reiher} on the graph $\overline G[X]$, i.e., the complement of $G$ induced on $X$. Put $\gamma = \frac12 - \frac{1}{2|X|} - \rho$, where $\rho$ is the density of edges in $G[X]$. Note that, first, $\gamma$ is exactly the density of $\overline G[X]$ and, second, $\frac{1}{2|X|} + \rho \le \frac{|X| / 2 + |e(G)|}{|X|^2} \le \frac{n/2 + mk}{|X|^2}$.

Recall that $\I'(\varnothing, F)$ is the family of independent $(k-\ell)$-sets in $G[X]$. Note that each set from $\I'(A',F)$ is the intersection (of size $k-\ell-|A'|$) of a $k$-element independent set in $G$ with $X$, and thus $|\I'(A', F)|$ is bounded from above by the number of independent $(k-\ell-|A|)$-sets in $G[X]$. Applying \eqref{reiher_eq}, we have
$$
|\I'(\varnothing, F)| \ge 
|\I'(A', F)| \cdot |X|^{|A'|} \prod_{i=0}^{|A'|-1} \frac{2(k-\ell-i-1)\gamma - (k-\ell-i-2)}{k-\ell-i} = 
$$
$$
=
|\I'(A', F)| \cdot |X|^{|A'|} \prod_{i=0}^{|A'|-1} \frac{1 - 2(k-\ell-i-1)(\rho + 1/2|X|)}{k-\ell-i}.
$$

\begin{cla}
We have $1 - 2(k-\ell-i-1)(\rho + 1/2|X|) \ge \frac{\eps k^2}{2|X|}$.
\end{cla}
\begin{proof}
Indeed,
$$
1 - 2(k-\ell-i-1)(\rho + 1/2|X|) \ge
\frac{(n - k(\ell+1))^2 - 2(n/2 + mk)(k-\ell-i-1)}{|X|^2} \ge
$$
$$
\frac{n^2 - O(nk^{3/2}) - 2mk(k-1)}{n|X|} \ge
\frac{\eps n k^2 / 2}{n|X|} =
\frac{\eps k^2}{2|X|}.
$$
\end{proof}

Then we have
$$
|\I'(\varnothing, F)| \ge 
|\I'(A', F)| \cdot |X|^{|A'|} \prod_{i=0}^{|A'|-1} \frac{\eps k^2}{2|X|(k-\ell-i)} \ge 
|\I'(A', F)| \cdot ( \eps k / 2 )^{|A'|}
$$

Recall that $|F| = t\le k$. Rewriting, we get

$$
\sum_{\varnothing\ne A'\subset F} \frac{|\I'(A', F)|}{|\I'(\varnothing, F)|} \le
\sum_{\varnothing\ne A'\subset F} (\eps k/2)^{-|A'|} =
$$
$$
\le 
\sum_{a=1}^{t} \binom{t}{a} (\eps k/2)^{-a} =
(1 + 2 / (\eps k))^t - 1
\le e^{2/\eps}.
$$
We used $t\le k$ and $1+x\le e^x$ in the last inequality. So we can put $c_\ell := e^{-2/\eps}$ in \eqref{eqcl}.

\bigskip

Finally, from \eqref{c_product}, using that $1 + x \ge e^{x-x^2}$ for $|x| \le 1/2$, we have
$$
|\I| \ge |\ff|\cdot \prod_{\ell=0}^{\sqrt{k}-1} (1 + c_\ell) =
|\ff|\cdot (1 + e^{-2/\eps})^{\sqrt{k}} \ge
|\ff|\cdot \exp\big(\sqrt{k}\cdot (e^{-2/\eps} - e^{-4/\eps})\big) =
$$
$$
|\ff|\cdot \exp\big(\sqrt{k}\cdot (k^{-0.4} - k^{-0.8})\big) \ge
|\ff|\cdot m,
$$
where the last inequality holds for $k$ large enough and $m < e^{k^{0.1} / 2}$. This completes the proof.
\end{proof}

\begin{small}

\end{small}

\end{document}